\documentclass{sig-alternate-per}
\usepackage{graphicx} 

\title{Low-Rank KKT Updates and a Parallel Flipping Mechanism for Model-Based Derivative-Free Optimization}
\author{Donghan Wu\orcidlink{0009-0005-8095-7949}\thanks{University of California, Berkeley.}, Pengcheng Xie\orcidlink{0000-0001-5973-1535}\thanks{Lawrence Berkeley National Laboratory. Corresponding author  (pxie@lbl.gov).}}
\date{April 2026}

\usepackage{amsmath,amssymb,amsfonts}
\usepackage{graphicx}
\usepackage{color}
\usepackage{booktabs}
\usepackage{algorithm}
\usepackage{algpseudocode}
\usepackage{bm}
\usepackage{tikz}
\usetikzlibrary{shapes.geometric, arrows.meta, positioning, fit, backgrounds, shadows}
\newtheorem{theorem}{Theorem}

\newcommand{\algo}[1]{\textsf{#1}}

\usepackage{orcidlink}

\begin{document}

\maketitle



{

\section{Introduction}
\label{Introduction}

\vspace{0.185cm}
In many real-world settings, optimization objective functions can be expensive to evaluate and do not have readily available derivatives. These problems arise when objective functions are expressed not as an algebraic function, but by querying a ``black box'' (or ``zeroth-order oracle''). Black boxes include data from the realization of a chemical process or output from a computer simulation and are typically addressed by derivative-free optimization algorithms \cite{LMW2019AN}. In this paper, we address such unconstrained problems of the form $\min_{\boldsymbol{x}\in\mathbb{R}^n} f(\boldsymbol{x})$, where the objective function $f$ is a deterministic black box possessing some degree of smoothness. 
Derivative-free algorithms for such problems are reviewed in \cite{AudetHare2017}. The model-based methods discussed in this paper typically use a trust-region framework for selecting new iteration points \cite{MJDP06, Rinaldi2024}. One example of such models is polynomials, particularly underdetermined quadratic interpolation \cite{Ragonneau2024, xieyuannew, xie2023dfoto,xie2023twodimensional,xie2023linesearch,XIE2025116146}, which balances model richness and evaluation cost. 
When the problem dimension $n$ is large, constructing a full quadratic model requires $\frac{1}{2}(n+1)(n+2)$ degrees of freedom, which is often unaffordable. Therefore, underdetermined quadratic interpolation models using significantly fewer points have become a major research direction. In such cases, one must impose additional principles to uniquely select a model. A classical and influential choice is the least Frobenius norm updating model \cite{MJDP06}, which selects the new quadratic model by minimizing the change of the Hessian in Frobenius norm. This idea underlies successful solvers such as NEWUOA. More recently, alternative regularization principles have been proposed, including least $H^2$ norm updating models \cite{10.1093/imanum/drae106}, as well as trust-region-aware underdetermined models \cite{xieyuannew} and regional minimal updating ideas \cite{xie2025remuregionalminimalupdating}.

However, several core bottlenecks remain. First, after each interpolation-set modification, many methods reconstruct or refactorize the associated linear system from scratch. This can become prohibitively expensive when updates are frequent. Second, large-scale applications increasingly require parallel computing environments, yet model construction itself is often sequential and becomes the new bottleneck even when function evaluations are parallelized. Third, maintaining model quality while aggressively updating interpolation points is delicate. Geometric quality, poisedness, and robustness against outliers remain important issues. Fourth, modern applications demand broader adaptability. Emerging settings include model-driven subspace strategies for large-scale problems \cite{he2025modeldrivensubspaceslargescaleoptimization}. These trends call for scalable linear algebraic mechanisms inside DFO algorithms.

In this paper, we develop a new structured and parallelizable framework for model construction in derivative-free trust-region methods, centered around the least Frobenius norm updating model. 
\section{Least  Norm Model}

\vspace{0.185cm}

In the model-based derivative-free trust-region method, we use a quadratic interpolation model $Q_k$ to locally approximate the objective function $f$ over an interpolation set $\mathcal{X}_k=\{\boldsymbol{x}_1,\ldots,\boldsymbol{x}_m\}$, where $n+2 \le m \le \frac{1}{2}(n+1)(n+2)$.

\subsection{Model Formulation}

\vspace{0.185cm}

When $m$ is less than the degrees of freedom of a full quadratic model, the interpolation system is underdetermined. To uniquely determine $Q_k$, Powell proposed minimizing the Frobenius norm of the change in the Hessian matrix:
\begin{equation}\label{leastfrob}
\underset{Q}{\operatorname{\min}}\ \left\Vert\nabla^{2} Q-\nabla^{2} Q_{k-1}\right\Vert_{F}^{2} \quad \text{s.t.} \ \ Q(\boldsymbol{x}_i)=f(\boldsymbol{x}_i), \ \boldsymbol{x}_i \in \mathcal{X}_{k}.
\end{equation}
Let the index ``old'' denote the $k$-th step, and ``new'' denote the $(k+1)$-th step. We define the difference function $D(\boldsymbol{x})=Q_{\text{new}}(\boldsymbol{x})-Q_{\text{old}}(\boldsymbol{x})$. When replacing a point $\boldsymbol{x}_t \in \mathcal{X}_k$ with $\boldsymbol{x}_{\text{new}}$, $D(\boldsymbol{x})$ is obtained by solving:
\begin{equation}
\label{miniFrob}
\begin{aligned}
\underset{D}{\operatorname{\min}} \  &\left\Vert \nabla^2 D\right\Vert_F^2 \\ 
\text{s.t.} \ 
& D(\boldsymbol{x}_{i})=0 \ (i \neq t), \ \ D(\boldsymbol{x}_{\text{new}})=f(\boldsymbol{x}_{\text{new}})-Q_{\text{old}}(\boldsymbol{x}_{\text{new}}).
\end{aligned}
\end{equation}

\subsection{KKT System and Structure}
\vspace{0.185cm}
Based on the KKT conditions (with Lagrange multipliers $\lambda_j$), the quadratic function $D(\boldsymbol{x})$ can be written in the form $D(\boldsymbol{x})=c+\boldsymbol{x}^{\top} \boldsymbol{g}+\frac{1}{2} \sum_{j=1}^{m} \lambda_{j}(\boldsymbol{x}^{\top}\boldsymbol{x}_{j})^{2}$. After determining the parameters $\boldsymbol{\lambda}=(\lambda_1,\ldots,\lambda_m)^{\top} \in \mathbb{R}^m$, $c \in \mathbb{R}$ and $\boldsymbol{g} \in \mathbb{R}^{n}$, the system of linear equations holds as:
\begin{equation}
\label{BIG-(3.10)}
\begin{pmatrix}
\boldsymbol{A} & \boldsymbol{X}^{\top} \\
\boldsymbol{X} & \boldsymbol{0}
\end{pmatrix}
\begin{pmatrix}
\boldsymbol{\lambda}^{\top}, c, \boldsymbol{g}^{\top}
\end{pmatrix}^{\top}=
\begin{pmatrix}
\boldsymbol{r}^{\top}, 0,\ldots,0
\end{pmatrix}^{\top},
\end{equation}
where $\boldsymbol{0}\in\mathbb{R}^{(n+1)\times(n+1)}$. The elements of $\boldsymbol{A}$ and $\boldsymbol{X}$ are strictly given by $\boldsymbol{A}_{i j} =\frac{1}{2}\left(\boldsymbol{x}_{i}^{\top}\boldsymbol{x}_{j}\right)^{2}$ and $\boldsymbol{X} = (\boldsymbol{1}, \boldsymbol{x}_{1}, \ldots, \boldsymbol{x}_{m})^\top$. The residual vector has the component $r_i=f(\boldsymbol{x}_{i})-Q_{\text{old}}(\boldsymbol{x}_{i})$. We define the symmetric KKT matrix $\boldsymbol{W}\in\mathbb{R}^{p\times p}$ ($p=m+n+1$) and its inverse $\boldsymbol{H}$ as $\boldsymbol{W}=\left(\begin{smallmatrix} \boldsymbol{A} & \boldsymbol{X}^{\top} \\ \boldsymbol{X} & \boldsymbol{0} \end{smallmatrix}\right)$ and $\boldsymbol{H}=\boldsymbol{W}^{-1}$. The updating parameters are directly yielded by $(\boldsymbol{\lambda}^{\top},c,\boldsymbol{g}^{\top})^{\top} =\boldsymbol{H}(\boldsymbol{r}^{\top},0,\ldots,0)^{\top}$. A critical property of $\boldsymbol{W}$ is that its sub-blocks rely strictly on the inner products of coordinates. This algebraic structure ensures that any single-point replacement or coordinate flipping transforms $\boldsymbol{W}$ via low-rank perturbations, allowing $\boldsymbol{H}$ to be updated efficiently without $\mathcal{O}(m^3)$ matrix inversions.

\section{Low-Rank Updates of KKT System}
\vspace{0.185cm}
\label{sec:low_rank}

Solving the KKT system directly requires $\mathcal{O}(m^3)$ operations. However, there is a low-rank updating formula for the KKT inverse matrix $\boldsymbol{H}$ if the change in $\boldsymbol{W}$ occurs strictly on specific rows or columns. The KKT matrix $\boldsymbol{W}$ for minimizing the Frobenius norm is updated via a Rank-2 formula in two primary cases: when a single point is replaced, or when all interpolation points undergo a specific transformation.

\subsection{Single Point Update and Flipping}
\vspace{0.185cm}
When a single interpolation point $\boldsymbol{x}_t$ is updated to a new point in the feasible region, the changed row and column of $\boldsymbol{W}$ are strictly the $t$-th ones, according to the definition of $\boldsymbol{W}$. Because the perturbation $\Delta\boldsymbol{W}$ is symmetric and confined exclusively to a single row and column, it can be analytically decomposed into the sum of two outer products ($\boldsymbol{e}_t \boldsymbol{w}^\top + \boldsymbol{w} \boldsymbol{e}_t^\top$). This guarantees a standard Rank-2 updating scenario. A far more powerful transformation arises when we wish to update \textit{all} interpolation points simultaneously to explore a new direction.
\begin{theorem}
\begin{sloppypar}
Assume that every interpolation point $\boldsymbol{x}_1, \dots, \boldsymbol{x}_m$ is updated to $\boldsymbol{x}_i + \alpha_i \boldsymbol{d}$, where $\boldsymbol{d} \in \mathbb{R}^n$ is a direction vector and $\alpha_i$ are scalars. If $\boldsymbol{d} = \boldsymbol{e}_t$ (the $t$-th standard basis vector) and $\alpha_i = -2\boldsymbol{x}_i^{(t)}$, the geometry sub-matrix $\boldsymbol{A}$ remains completely unchanged, and the overall perturbation to the KKT matrix $\boldsymbol{W}$ is strictly of Rank-2.
\end{sloppypar}
\end{theorem}

\begin{proof}
Let the updated coordinate matrix be $\widehat{\boldsymbol{X}} = [1, \dots, 1; \boldsymbol{x}_1 + \alpha_1 \boldsymbol{d}, \dots, \boldsymbol{x}_m + \alpha_m \boldsymbol{d}]^\top$. The updated elements of the geometry matrix $\boldsymbol{A}$ are:
\begin{equation}
\begin{aligned}
\hat{\boldsymbol{A}}_{i j} &=\frac{1}{2}\left((\boldsymbol{x}_i + \alpha_i \boldsymbol{d})^\top (\boldsymbol{x}_j + \alpha_j \boldsymbol{d})\right)^2 \\
&=\frac{1}{2}\left(\boldsymbol{x}_i^\top \boldsymbol{x}_j + \alpha_i \alpha_j + \alpha_j \boldsymbol{x}_i^\top \boldsymbol{d} + \alpha_i \boldsymbol{x}_j^\top \boldsymbol{d}\right)^2.
\end{aligned}
\end{equation}
To ensure that the quadratic geometry is preserved, we require $\hat{\boldsymbol{A}}_{ij} = \boldsymbol{A}_{ij} = \frac{1}{2}(\boldsymbol{x}_i^\top \boldsymbol{x}_j)^2$. This condition holds if and only if the inner terms satisfy $\alpha_i \alpha_j + \alpha_j \boldsymbol{x}_i^\top \boldsymbol{d} + \alpha_i \boldsymbol{x}_j^\top \boldsymbol{d} = 0$, or alternatively, $\alpha_i \alpha_j + \alpha_j \boldsymbol{x}_i^\top \boldsymbol{d} + \alpha_i \boldsymbol{x}_j^\top \boldsymbol{d} = -2\boldsymbol{x}_i^\top \boldsymbol{x}_j$. However, if $i=j$ and $\boldsymbol{d} = \boldsymbol{e}_t$ (where only the $t$-th component is 1), the second alternative implies $\alpha_i^2 + 2\alpha_i \boldsymbol{x}_i^{(t)} + 2\boldsymbol{x}_i^\top \boldsymbol{x}_i = 0$. Since the discriminant $4(\boldsymbol{x}_i^{(t)})^2 - 8(\boldsymbol{x}_i^\top \boldsymbol{x}_i) < 0$ for $\boldsymbol{x}_i \neq 0$, this equation has no real roots. Therefore, we can only obtain valid scalars $\alpha_i$ by solving the first condition. Substituting $\boldsymbol{d}=\boldsymbol{e}_t$ into the first condition when $i=j$ yields $\alpha_i^2 + 2\alpha_i \boldsymbol{x}_i^{(t)} = 0$. The non-trivial solution to this constraint is $\alpha_i = -2\boldsymbol{x}_i^{(t)}$ for all $i=1, \dots, m$. 

Substituting $\alpha_i = -2\boldsymbol{x}_i^{(t)}$ into the point update formula gives $\boldsymbol{x}_i - 2\boldsymbol{x}_i^{(t)} \boldsymbol{e}_t$. This operation strictly negates the $t$-th coordinate of every point $\boldsymbol{x}_i$, which we define as a \textit{$t$-axis flip}. Because $\alpha_i$ satisfies the required condition, the geometry block $\boldsymbol{A}$ remains completely unchanged. The perturbation $\Delta\boldsymbol{W} = \hat{\boldsymbol{W}} - \boldsymbol{W}$ occurs solely due to the updated coordinate block $\widehat{\boldsymbol{X}}$, which differs from $\boldsymbol{X}$ exactly in its $(t+1)$-th row (corresponding to the $t$-th spatial dimension). Thus, the non-zero elements of the symmetric matrix $\Delta\boldsymbol{W}$ are strictly confined to the $k$-th row and column, where $k = m+1+t$. Such a cross-shaped perturbation can be algebraically decomposed into a strict symmetric Rank-2 form:
\begin{equation}
\label{eq:rank2_decomp}
\Delta\boldsymbol{W} = \boldsymbol{e}_k \boldsymbol{w}^\top + \boldsymbol{w} \boldsymbol{e}_k^\top,
\end{equation}
where $\boldsymbol{e}_k$ is the $k$-th standard basis vector, and the vector $\boldsymbol{w}$ incorporates the $k$-th column of $\Delta\boldsymbol{W}$ (with its $k$-th diagonal element halved).
\end{proof}

\subsection{Rank-2 Updating Mechanism}
\vspace{0.185cm}
Given the explicit Rank-2 decomposition $\Delta\boldsymbol{W} = \boldsymbol{e}_k \boldsymbol{w}^\top + \boldsymbol{w} \boldsymbol{e}_k^\top$ established above, the new inverse matrix $\boldsymbol{H}_{\text{new}} = (\boldsymbol{W} + \Delta\boldsymbol{W})^{-1}$ can be analytically updated without $\mathcal{O}(m^3)$ refactorization. Dropping the subscript $k$ for brevity, the Sherman-Morrison-Woodbury identity yields an $\mathcal{O}(n^2)$ update formula:
\begin{equation}
\label{updating-formula}
\begin{aligned}
\boldsymbol{H}_{\text{new}}=&\boldsymbol{H}+\frac{1}{\sigma}\Big[\alpha\left(\boldsymbol{e}-\boldsymbol{H} \boldsymbol{w}\right)\left(\boldsymbol{e}-\boldsymbol{H} \boldsymbol{w}\right)^{\top}-\beta \boldsymbol{H} \boldsymbol{e} \boldsymbol{e}^{\top} \boldsymbol{H}\\
&+\tau\left\{\boldsymbol{H} \boldsymbol{e}\left(\boldsymbol{e}-\boldsymbol{H} \boldsymbol{w}\right)^{\top}+\left(\boldsymbol{e}-\boldsymbol{H} \boldsymbol{w}\right) \boldsymbol{e}^{\top} \boldsymbol{H}\right\}\Big],
\end{aligned}
\end{equation}
where $\alpha=\boldsymbol{e}^{\top} \boldsymbol{H} \boldsymbol{e}$, $\beta=(\boldsymbol{e}-\boldsymbol{H} \boldsymbol{w})^{\top} \boldsymbol{w}$, $\tau=\boldsymbol{e}^{\top} \boldsymbol{H} \boldsymbol{w}$, and $\sigma=\alpha \beta+\tau^{2}$.
\section{Algorithm}
\vspace{0.185cm}

Based on the flipping mechanism, we distribute the initial model to $P$ parallel machines. Each machine randomly flips a coordinate axis, instantly updates its local inverse $\mathbf{H}^{(i)}$ using Eq. \eqref{updating-formula}, and solves the trust-region subproblem via Truncated Conjugate Gradient (TCG). The framework is summarized in Algorithm \ref{alg:framework}. To intuitively illustrate the execution flow, Figure \ref{fig:algo_flowchart} provides a schematic overview of the parallel framework.

\begin{algorithm}[h!]
\caption{Parallel Trust-Region with Flipping
}
\label{alg:framework}
\small
\begin{algorithmic}[1]
\Require dimension $n$, machines $P$, base point $\boldsymbol{x}_{0}$, steps $S=10$
\State \textbf{Init:} $\mathcal{X} \gets 2n+1$ points around $\boldsymbol{x}_{0}$. $Q \gets 0$, $\boldsymbol{H} \gets \boldsymbol{W}_0^{-1}$, $\Delta \gets \Delta_{\text{init}}$.
\While{termination criteria not met}
    \For{machine $i=1,\dots,P$ \textbf{in parallel}}
        \State \textbf{[Flip]} $t_i \gets \text{rand}\{1..n\}$; \ $\hat{\mathcal{X}}^{(i)} \gets$ flip $t_i$-th axis of $\mathcal{X}$.
        \State \textbf{[Update]} Extract $t_i$-th coords into $\boldsymbol{w}$; \ $\boldsymbol{H}^{(i)} \gets \text{Rank2Up}(\boldsymbol{H}, t_i\!+\!m\!+\!1, \boldsymbol{w})$.
        \State Solve KKT for parameters; \ Update local model $Q^{(i)}$.
        
        \State \textbf{[Inner Loop]} $\boldsymbol{H}_s, Q_s \gets \boldsymbol{H}^{(i)}, Q^{(i)}$; \ $\Delta^{(i)} \gets \Delta$.
        \For{$s=1,\dots,S$}
            \State $\boldsymbol{x}_{\text{opt}} \gets \arg\min_{\boldsymbol{x} \in \hat{\mathcal{X}}^{(i)}} f(\boldsymbol{x})$.
            \State $\boldsymbol{x}_{\text{new}} \gets \boldsymbol{x}_{\text{opt}} + \text{TCG}(\nabla Q_s(\boldsymbol{x}_{\text{opt}}), \nabla^2 Q_s, \Delta^{(i)})$.
            \State $\rho \gets (f_{\text{old}} - f_{\text{new}}) / (Q_{\text{old}} - Q_{\text{new}})$.
            \State Update trust-region radius $\Delta^{(i)}$ based on $\rho$.
            \If{$\rho > 0$ \textnormal{(Step accepted)}}
                \State Replace worst point $t$ in $\hat{\mathcal{X}}^{(i)}$ with $\boldsymbol{x}_{\text{new}}$.
                \State Extract new coords into $\boldsymbol{w}$; \ $\boldsymbol{H}_s \gets \text{Rank2Up}(\boldsymbol{H}_s, t, \boldsymbol{w})$.
                \State Update local model $Q_s$ via new KKT residuals.
            \EndIf
        \EndFor
        \State Save local best $\boldsymbol{x}_{\text{opt}}^{(i)}$, and matrix state $\boldsymbol{H}^{(i)} \gets \boldsymbol{H}_s$.
    \EndFor
    \State \textbf{[Sync]} $\text{best} \gets \arg\min_i f(\boldsymbol{x}_{\text{opt}}^{(i)})$.
    \State Broadcast states and $\Delta$ of machine \text{best} as new global baseline.
\EndWhile
\end{algorithmic}
\end{algorithm}


\begin{figure}[htbp]
    \centering
    \resizebox{0.5\columnwidth}{!}{%
    \begin{tikzpicture}[
        node distance = 0.3cm, 
        box/.style = {rectangle, rounded corners, draw=black!70, thick, fill=white, drop shadow={opacity=0.15}},
        globalnode/.style = {box, fill=blue!10, text width=0.8\columnwidth, align=center, inner sep=3pt}, 
        workernode/.style = {box, fill=orange!10, text width=0.8\columnwidth, align=left, inner sep=3pt},
        syncnode/.style = {box, fill=gray!15, text width=0.8\columnwidth, align=center, inner sep=3pt},
        arrow/.style = {thick, -{Latex[scale=1.0]}}
    ]

    \node (global) [globalnode] {
        \textbf{Global State $k$} \\
        Base: $\mathbf{x}_{\text{best}}^{(k)}$, Inv: $\mathbf{H}^{(k)}$, Rad: $\Delta^{(k)}$
    };

    \node (w_back2) [workernode, below=0.5cm of global, xshift=3pt, yshift=-3pt] {};
    \node (w_back1) [workernode, below=0.5cm of global, xshift=1.5pt, yshift=-1.5pt] {};
    
    \node (worker) [workernode, below=0.5cm of global] {
        \centering \textbf{Parallel Execution ($P$ Machines)} \\[1mm]
        \textbf{1. Flip:} Select $t_i \sim \mathcal{U}\{1..n\}$ and flip axis. \\
        \textbf{2. Update:} $\mathcal{O}(n^2)$ Rank-2 KKT update for $\boldsymbol{H}^{(i)}$. \\
        \textbf{3. Search:} Inner TR search via TCG solver.
    };

    \node (sync) [syncnode, below=0.5cm of worker] {
        \textbf{Global Synchronization} \\
        Select best state: $i^* = \arg\min_i f(\boldsymbol{x}_{\text{opt}}^{(i)})$
    };

    \draw [arrow] (global) -- node[right, font=\footnotesize] {Broadcast} (worker);
    \draw [arrow] (worker) -- node[right, font=\footnotesize] {Gather} (sync);
    
    \draw [arrow, rounded corners=6pt, color=red!70!black] 
        (sync.west) -- ++(-0.3,0) |- node[near start, left, font=\bfseries\footnotesize, text=red!70!black, align=center, text width=1.5cm] {Update\\Global\\Baseline} (global.west);

    \end{tikzpicture}%
    } 
    
    \vspace{-2mm}
    \caption{Schematic overview of the parallel flipping trust-region framework. }
    \label{fig:algo_flowchart}
\end{figure}
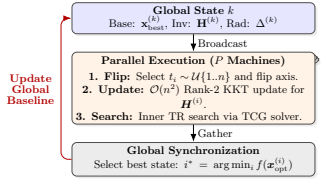

\section{Numerical Experiments}
\label{sec:experiments}
\vspace{0.185cm}
To evaluate the parallel flipping mechanism, we conduct experiments on the classic derivative-free benchmark set by Mor\'{e} and Wild \cite{JJMSMW09}. The set comprises 53 unique problems ($n \in [2, 12]$) with 10 variations (e.g., smooth, non-differentiable, and noisy with variance $10^{-4}$), totaling 530 instances. We aggregate the results using data profiles $\delta_a(\beta)$ and performance profiles $\rho_a(\alpha)$ \cite{EDD01, JJMSMW09}.

An algorithm successfully solves an instance if the normalized reduction satisfies:
\begin{equation}
\label{eq:f_acc}
f_{\mathrm{acc}}^{N} := \frac{f(\boldsymbol{x}_{0}) - f(\boldsymbol{x}_{N})}{f(\boldsymbol{x}_{0}) - f(\boldsymbol{x}_{\text{best}})} \ge 1-\tau,
\end{equation}
where $\boldsymbol{x}_{N}$ is the best value found after $N$ evaluations, $\boldsymbol{x}_{0}$ is the starting point, and $\boldsymbol{x}_{\text{best}}$ is the known optimum. We evaluate four tolerance levels: $\tau \in \{10^{-1}, 10^{-2}, 10^{-3}, 10^{-4}\}$.  
We test our algorithm with varying parallel machines $P \in \{2, 4, 8\}$ against two baseline solvers: \algo{Large-Scale} and \algo{QARSTA}. The initial trust-region radius is $\Delta_0 = \max \{1, \|\boldsymbol{x}_{0}\|_{\infty}\}$, and updating parameters are set to $\gamma=2, \eta_1=0.25, \eta_2=0.75$.

Executed on Windows with an Intel Core Ultra 9 285H CPU and 32GB RAM. 
The results presented in Figures~\ref{fig:data_profiles} 
provide a comprehensive comparison of the tested methods across a wide range of accuracy levels. The data profiles in Figure~\ref{fig:data_profiles} illustrate the fraction of problems solved within a given normalized computational budget. We observe that our proposed method consistently achieves a higher success rate across all tolerance levels, especially in the low-budget regime. 



\begin{figure}[t!]
    \centering
    \includegraphics[width=0.6\columnwidth, trim=0 0 0 0, clip]{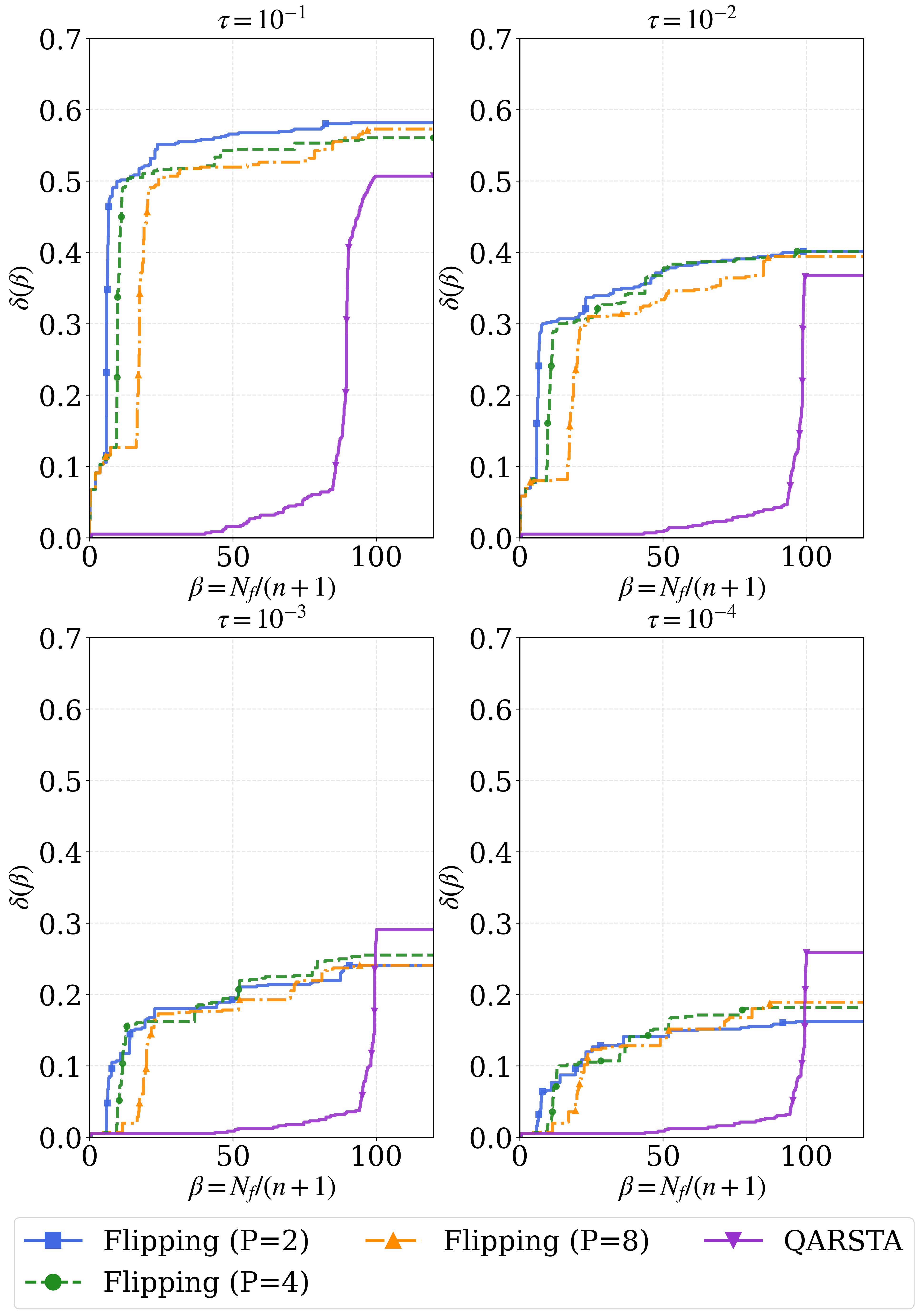}
    \vspace{-4mm} 
    \caption{Data profiles $\delta(\beta)$ for the 530 test problems under varying accuracy tolerance levels $\tau$. The x-axis indicates the normalized budget $\beta = N_f/(n+1)$.}
    \label{fig:data_profiles}
\end{figure}



\section{Conclusion}
\vspace{0.185cm}
We developed a new parallel framework for model-based derivative-free optimization by combining least Frobenius norm updating models, structured low-rank KKT updates, and a flipping-based interpolation mechanism. The central contribution of this work is to show that interpolation-set modifications in the least Frobenius norm setting are not merely implementation details, but possess exploitable algebraic structure that can be used to substantially reduce model-maintenance cost.




\bibliographystyle{acm}

\begingroup
\vspace{1em}   
\small
\bibliography{bigrefs,xpub}
\endgroup



\end{document}